\documentclass[reqno,11pt]{amsart}

\usepackage[colorlinks=true,urlcolor=blue,
citecolor=red,linkcolor=blue,linktocpage,pdfpagelabels,
bookmarksnumbered,bookmarksopen]{hyperref}

\usepackage[left=2.6cm,right=2.6cm,top=2.9cm,bottom=2.9cm]{geometry}

\usepackage{amsmath,amssymb,latexsym,soul,cite,mathrsfs}
\numberwithin{equation}{section}  
\usepackage{color,enumitem,graphicx}
\usepackage{enumerate}
\usepackage{cancel}
\usepackage{mathrsfs}

\usepackage[colorlinks=true,urlcolor=blue,
citecolor=red,linkcolor=blue,linktocpage,pdfpagelabels,
bookmarksnumbered,bookmarksopen]{hyperref}
\usepackage[english]{babel}

\newtheorem{theorem}{Theorem}[section]
\newtheorem{proposition}[theorem]{Proposition}
\newtheorem{corollary}[theorem]{Corollary}
\newtheorem{lemma}[theorem]{Lemma}
\DeclareMathAlphabet{\mathpzc}{OT1}{pzc}{m}{it}

\newtheorem{example}[theorem]{Example}

\newtheorem{remark}[theorem]{Remark}
\newcommand{\Ric}{\operatorname{Ric}}
\newtheorem{theoremletter}{Theorem}

\title[Torsional Rigidity and Spherical Deficit for a Dirichlet Problem]{Torsional Rigidity and Spherical Deficit for a Dirichlet Problem on Riemannian Manifolds}

\author[M. Andrade]{Maria Andrade}
\address[M. Andrade]{Universidade Federal de Sergipe - UFS, Departamento de Matemática, 49100-000, S\~ao Cristov\~ao-SE, Brazil.
}
\email{\href{mailto: maria@mat.ufs.br}{ maria@mat.ufs.br}}

\author[A. Freitas]{Allan Freitas}

\address[A. Freitas]{Universidade Federal da Paraíba - UFPB, Departamento de Matemática, João Pessoa, PB 58051-900, Brazil}

\address{Current address: University of São Paulo, Institute of Mathematics, Statistics and Computer Science,  05508-090, São Paulo-SP, Brazil}

\email{\href{mailto: allan@mat.ufpb.br}{allan@mat.ufpb.br}; \href{mailto: allan.freitas@academico.ufpb.br}{allan.freitas@academico.ufpb.br}; \href{mailto: allan.freitas@ime.usp.br}{allan.freitas@ime.usp.br}}


\begin{document}

\begin{abstract}
In this work, we study several inequalities related to a Dirichlet problem on Riemannian manifolds whose Ricci curvature is bounded from below. First, we establish sharp inequalities involving the torsional rigidity and show that equality forces the domain to be isometric to a geodesic ball in a space form, with no symmetry assumption on the ambient manifold. Next, we derive an integral identity associated with the same Dirichlet problem, which measures the spherical deficit arising in this context. In particular, we apply this identity to Einstein manifolds, where it yields a rigidity result and a bound for the weighted spherical deficit in terms of the deviation of the normal derivative from a constant.
\end{abstract}

\maketitle




\section{Introduction}\label{sec:intro}
Let $\Omega$ be a bounded domain in the Euclidean space $\mathbb{R}^n$ with non-empty boundary $\partial\Omega$, and let the \textit{torsion function} $u : \Omega \to \mathbb{R}$ be the unique solution of
\begin{equation}\label{Dirichlet1}
\Delta u = -1\quad \text{in } \Omega, \quad u = 0 \text{ on } \partial \Omega.
\end{equation}
In this setting, the \textit{torsional rigidity} is the functional defined by
\begin{equation}
T(\Omega) = \int_{\Omega} u \, dv.
\end{equation}
The importance, applications, and meaning of both the torsion function and the torsional rigidity can be seen in many areas of mathematics, such as elasticity theory (see \cite{polya1951,sokolniko1956}), fluid dynamics (see \cite{serrin1971}), heat conduction (see \cite{vandenberg1990}), and also in the study of minimal submanifolds \cite{markvorsen2006}.
Just to mention an example, in the $n=2$ case, $T(\Omega)$ represents the torque required for a unit angle of twist per unit
length when twisting an elastic beam of uniform cross section $\Omega$ (see \cite[Pages~109-119]{sokolniko1956}).

In the case where an additional boundary condition is imposed together with \eqref{Dirichlet1}, namely
\begin{equation}\label{neumann}
u_{\nu} = -c,
\end{equation}
where $c$ is a positive constant, we are faced with a so-called overdetermined problem. This type of problem was first studied by J. Serrin in \cite{serrin1971}, who proved that the only solution of \eqref{Dirichlet1}--\eqref{neumann} is a Euclidean ball, and consequently the torsion function $u$ is radial. In an interesting application to elasticity theory, Serrin's celebrated result can be paraphrased as follows \cite{serrin1971}: \textit{``when a solid straight bar is subject to torsion, the magnitude of the resulting traction which occurs at the surface of the bar is independent of position if and only if the bar has a circular cross section.''}

In both the Euclidean setting, which has been primarily studied, and in the context of domains in \textit{Riemannian manifolds}, the analysis of problems of the type \eqref{Dirichlet1} has attracted considerable attention. A particularly interesting and recently explored direction is to establish bounds for the torsional rigidity $T(\Omega)$. In this regard, Wang and Xia (see \cite{wang2020}) have recently obtained sharp estimates for the torsional rigidity where $\Omega=(M^n,g)$ is a compact Riemannian manifold with boundary. In one of their results they proved
\begin{theorem}[Theorem 1.7 in \cite{wang2020}]
Let $(M^n,g)$ be an $n$-dimensional compact Riemannian manifold with boundary and Ricci curvature $Ric\geq(n-1)kg, \ k\in\mathbb{R}$. Let $u$ be the solution of the Dirichlet problem \eqref{Dirichlet1}.
Then
\begin{equation}\label{bound}
\max_{\partial M} |\nabla u|^2
\geq
\frac{(n+2)T(M)}{nVol(M)}
+ \frac{2(n-1)k}{Vol(M)} \int_M u\,|\nabla u|^2dv,
\end{equation}
with equality holding if and only if $k = 0$ and $M$ is isometric to a ball in $\mathbb{R}^n$.
\end{theorem}
In the Euclidean case, and in the presence of the boundary condition \eqref{neumann}, the previous result provides an alternative proof of Serrin's classical theorem. Indeed, in this particular case, equality in \eqref{bound} follows from the classical Pohozaev identity \cite{pohozaev1965} (see also \cite[p.~171]{struwe}), and the corresponding rigidity result then follows (see \cite{wang2020} for details).

In the context of Riemannian manifolds, particularly those whose Ricci curvature satisfies $\mathrm{Ric} \geq (n-1)k g, \ k\in\mathbb{R},$ the following Dirichlet problem has attracted considerable attention in recent years:
\begin{equation}\label{Dirichlet2}
\Delta u + n k u = -1 \quad \text{in } \Omega,
\qquad
u = 0 \quad \text{on } \partial \Omega.
\end{equation}
This problem is typically studied in conjunction with the overdetermined condition \eqref{neumann}. In the case where $k$ corresponds to the sectional curvature of an $n$-dimensional space form, Ciraolo and Vezzoni \cite{ciraolo} extended Weinberger's method and established the corresponding rigidity result in $\mathbb{S}^{n}_{+}$ (when $k=1$) and $\mathbb{H}^{n}$ (when $k=-1$) by means of a suitable $P$-function. Subsequently, under the same lower bound on the Ricci curvature, this rigidity phenomenon has also been observed in certain classes of warped products (see \cite{Roncoroni2018, farina2022}) and in manifolds endowed with a conformal vector field (see \cite{freitas2024, andrade2025}).

The first goal of this paper is to establish new estimates for the torsional rigidity associated with problem \eqref{Dirichlet2}. To this end, we employ the classical $P$-function in this setting (see Section~\ref{sec_2}), extending the approach developed in \cite{wang2020}. A distinctive feature of problem \eqref{Dirichlet2}, in contrast with \eqref{Dirichlet1}, is that the equality cases are not restricted to $k=0$: combining the $P$-function with a recent Obata-type lemma of Accornero and Ciraolo \cite[Lemma 2.3]{accornero}, we show that equality characterizes geodesic balls in the space form of curvature $k$, for every $k\in\mathbb{R}$ and without any symmetry assumption on the ambient manifold. Throughout the paper, $\Omega$ is a domain with compact closure and $C^{2}$ boundary, and $u\in C^{2}(\overline{\Omega})$ is a solution of \eqref{Dirichlet2} which is positive in $\Omega$ (this is automatic when $k\leq 0$, by the maximum principle).

\begin{theoremletter}\label{MainTh_1}
Let $(M^{n},g)$ be a Riemannian manifold whose Ricci curvature satisfies $\operatorname{Ric} \geq (n-1)k g$, $k\in\mathbb{R}$, and let $\Omega \subset M$ be a domain with compact closure and $C^{2}$ boundary. Let $u$ be a solution of \eqref{Dirichlet2} with $u>0$ in $\Omega$. Then
\begin{equation}\label{eqtor}
\max_{\partial\Omega}|\nabla u|^2 \;\geq\; \frac{1}{\operatorname{Vol}(\Omega) + nk\,T(\Omega)}\;\cdot\;\frac{n+2}{n}\left[ T(\Omega) + nk\int_{\Omega}\bigl(u^2 + u|\nabla u|^2\bigr)dv \right].
\end{equation}
Moreover, equality holds in \eqref{eqtor} if and only if $\Omega$ is isometric to a geodesic ball in the simply connected space form of constant curvature $k$. In this case, $u$ is radial.
\end{theoremletter}

For $k=0$, Theorem \ref{MainTh_1} recovers \cite[Theorem 1.7]{wang2020} for manifolds with nonnegative Ricci curvature. For $k\neq 0$, the equation \eqref{Dirichlet2} is the one adapted to the curvature bound, and the inequality becomes sharp on geodesic balls of $\mathbb{H}^{n}$ and of $\mathbb{S}^{n}_{+}$, which is not the case for \eqref{bound}.

The second part of this paper is devoted to deriving an integral identity in order to analyze the spherical deficit associated with the solutions of \eqref{Dirichlet2}. The motivation for this goal comes from \cite{poggesi,poggesi1}, where the authors established new integral identities yielding quantitative estimates for radially symmetric configurations associated with the problem \eqref{Dirichlet1}. In this direction, for both Alexandrov's Soap Bubble Theorem and Serrin's overdetermined problem, the quantity
\[
|\mathring{\nabla}^{2}u|^{2}
=
|\nabla^{2}u|^{2}
-
\frac{(\Delta u)^{2}}{n}
\]
plays a fundamental role in measuring how close a domain is to being spherical, and for this reason we refer to it as the \textit{spherical (or Cauchy--Schwarz) deficit}. Indeed, the condition
\[
\mathring{\nabla}^{2}u=0
\]
leads to rigidity phenomena both in the Euclidean and in the Riemannian settings. In the Euclidean case, it characterizes quadratic polynomials, while in the Riemannian setting, Obata-type theorems identify metric balls as the only possible solutions (see, for example, \cite{obata, reilly}, \cite[Lemma 6]{farina2022} and \cite[Lemma 2.3]{accornero}). In the next result, $q$ is a suitable quadratic polynomial solution for \eqref{Dirichlet1} and
$$c=\frac{\operatorname{Vol}(\Omega)}{\operatorname{Area}(\partial\Omega)}.$$

\begin{theorem}[Theorem 2.1 in \cite{poggesi}]
Let $\Omega \subset \mathbb{R}^{n}$ be a bounded domain with boundary $\partial\Omega$ of class $C^{1,\alpha}$, $0<\alpha\leq 1$, and let $c$ be the constant defined above. Then the solution $u$ of \eqref{Dirichlet1} satisfies the identity
\begin{equation}\label{eq:serrin_identity}
-2\int_{\Omega} u|\mathring{\nabla}^{2}u|^2dv
=\int_{\partial\Omega}(u_{\nu}^{2}-c^{2})(u_{\nu}-q_{\nu})d\sigma.
\end{equation}
In particular, if the right-hand side of \eqref{eq:serrin_identity} is non-negative, then $\Omega$ must be a ball of radius $R=nc$. This is the case whenever $u_{\nu}$ is constant on~$\partial\Omega$.
\end{theorem}

The main feature of this type of identity is that a boundary condition, namely, the constancy of the normal derivative of $u$, yields rigidity of the domain $\Omega$. Our second main result is devoted to establishing an integral identity in this direction for general Riemannian manifolds endowed with a conformal vector field (see Section~\ref{sec_2} for further details). Although we are able to derive a general integral identity for solutions of \eqref{Dirichlet2}, as stated in Proposition~\ref{intlemma}, which allows us to analyze the role played by the Ricci and scalar curvatures in these identities, we emphasize here its particularly intriguing application to the setting of Einstein manifolds, which we describe below. In this special case, the existence of solutions is guaranteed, for example, by \cite[Theorem 2.3]{foga} (see Remark \ref{existence} for details).

In what follows, $\varphi$ denotes a \emph{reference solution}, that is, a function satisfying
\[
\nabla^{2}\varphi=-\Bigl(\frac1n+k\varphi\Bigr)g \quad\text{in }\Omega,
\]
which plays the role of the quadratic polynomial $q$ of the Euclidean case.

\begin{theoremletter}\label{MainTh_2}
Let $(M^n,g)$, $n\geq 3$, be an Einstein Riemannian manifold with $\Ric=(n-1)kg, \ k\in\mathbb{R}$, endowed with a conformal vector field $X$ with conformal factor $V$, and let $\Omega\subset M$ be a domain with compact closure and $C^{2}$ boundary. Let $\varphi$ be a reference solution and let $u$ be a solution of \eqref{Dirichlet2}. Then, for every constant $c\in\mathbb{R}$,
\begin{equation}\label{new:idB}
-2\int_{\Omega}Vu\,|\mathring{\nabla}^{2}u|^{2}dv
=
\int_{\partial\Omega}(u_{\nu}^{2}-c^{2})
\bigl[V(u-\varphi)_{\nu}-(u-\varphi)V_{\nu}\bigr]d\sigma.
\end{equation}
\end{theoremletter}

\begin{remark}\label{new:remB}
\begin{itemize}[leftmargin=*]
\item[(i)] In the Euclidean setting, where the position vector field $x$ is conformal with conformal factor $V\equiv 1$, identity \eqref{new:idB} reduces to the classical identity established in \cite[Theorem 2.1]{poggesi}.
\item[(ii)] The restriction $n\geq 3$ is only used to guarantee $\mathring{\nabla}^{2}V=0$ (Lemma \ref{lemma_tracefree}). The theorem remains true for $n=2$ if this is assumed, for instance when $X$ is closed conformal.
\item[(iii)] If $V>0$ and $u>0$ in $\Omega$, then \eqref{new:idB} yields the following bound for the weighted spherical deficit in terms of the deviation of the flux from a constant:
\[
0\leq 2\int_{\Omega}Vu\,|\mathring{\nabla}^{2}u|^{2}dv\leq
\bigl\|V(u-\varphi)_{\nu}-(u-\varphi)V_{\nu}\bigr\|_{L^{\infty}(\partial\Omega)}\;\bigl\|u_{\nu}^{2}-c^{2}\bigr\|_{L^{1}(\partial\Omega)} .
\]
\item[(iv)] When $M$ is the hyperbolic space or a hemisphere and $\varphi$ is constant, identity \eqref{new:idB} coincides with an identity obtained by Scheuer and Xia \cite[(3.1)]{scheuerxia}; see Remark \ref{new:remSX}.
\end{itemize}
\end{remark}

\textbf{Organization of the paper.}
In Section~\ref{sec_2}, we recall some basic concepts and preliminary results that will be used throughout the paper. In Section~\ref{sec_estimates}, we investigate results related to torsional rigidity, establishing bounds and rigidity statements. In particular, we prove Theorem~\ref{MainTh_1}. Finally, in Section~\ref{sec_deficit}, we establish a general integral identity associated with problem \eqref{Dirichlet2}, extending the identity obtained in \cite[Theorem 2.1]{poggesi}; see Proposition~\ref{intlemma}. As an application, we prove Theorem~\ref{MainTh_2}.

\section{Preliminaries}\label{sec_2}

In this section, we gather the essential concepts and auxiliary results that will be used throughout the paper. Whenever appropriate, we highlight how each tool connects to the problems and goals outlined in the Introduction.

We begin by recalling Serrin's type problem in the Riemannian setting:
\begin{eqnarray}
\Delta u + nku  &=& -1 \quad \text{in} \quad \Omega, \quad u = 0 \quad \text{on} \quad \partial \Omega, \label{serrinoka}\\
u_{\nu} &=& -c, \quad \text{on} \quad \partial \Omega, \label{serrinokb}
\end{eqnarray}
where $c$ is a constant and $k\in\mathbb{R}$. As discussed in the Introduction, the parameter $k$ reflects a lower bound on the Ricci curvature of the ambient manifold, and the overdetermined condition \eqref{serrinokb} is the distinctive feature of Serrin-type rigidity phenomena.

Throughout the paper, $\Omega\subset M$ is a domain (open and connected) with compact closure and nonempty boundary of class $C^{2}$, and $u\in C^{2}(\overline{\Omega})$ is a solution of \eqref{serrinoka} such that
\begin{equation}\label{new:pos}
u>0 \quad\text{in }\Omega .
\end{equation}
When $k\leq 0$, condition \eqref{new:pos} is automatically satisfied, by the maximum principle. When $k>0$ it is a genuine assumption: for instance, it fails on geodesic balls of $\mathbb{S}^{n}$ of radius larger than $\pi/2$.

\begin{lemma}\label{new:lempos}
Under the standing assumptions, $1+nku>0$ in $\overline{\Omega}$. Consequently,
\[
\Delta u<0 \ \text{ in }\Omega,\qquad u_{\nu}<0\ \text{ on }\partial\Omega,\qquad \operatorname{Vol}(\Omega)+nk\,T(\Omega)=\int_{\Omega}(1+nku)\,dv>0 ,
\]
where $T(\Omega)=\int_{\Omega}u\,dv$.
\end{lemma}
\begin{proof}
If $k\geq0$ the first claim follows from \eqref{new:pos}. If $k<0$, the function $w=1+nku$ satisfies $\Delta w=-nk\,w$ in $\Omega$ and $w=1$ on $\partial\Omega$; since $-nk>0$, $w$ cannot attain a nonpositive minimum in $\Omega$ by the strong maximum principle (see also \cite[Lemma 5.2]{andrade2025}). The remaining assertions follow from $\Delta u=-(1+nku)$, Hopf's lemma and the divergence theorem.
\end{proof}

A key quantity associated with problem \eqref{serrinoka} is the torsional rigidity of $\Omega$. By integration by parts, one easily verifies that
\[
\int_{\Omega} |\nabla u|^2dv = \int_{\Omega} (-\Delta u)  u \, dv = \int_{\Omega} (1 + n k u)  u \, dv= \int_{\Omega}  u \, dv + n k \int_{\Omega} u^2 \,dv.
\]
This motivates the definition of torsional rigidity
\begin{equation}\label{eqtorsion}
T(\Omega) := \int_{\Omega} |\nabla u|^2\, dv - n k \int_{\Omega} u^2\, dv = \int_{\Omega} u \, dv.
\end{equation}
The quantity $T(\Omega)$ will be central in Section~\ref{sec_estimates}, where we establish sharp bounds for it under curvature assumptions.

A classical approach to Serrin's problem, both in Euclidean space and on Riemannian manifolds, employs an auxiliary function known as the $P$-function. In our setting, it is given by
\begin{equation}\label{pfunction}
P := |\nabla u|^2 + \dfrac{2}{n}u + ku^2.
\end{equation}
The importance of $P$ stems from the following fact: under the curvature assumption $\operatorname{Ric} \geq (n-1)k g$, the function $P$ is subharmonic for any solution of \eqref{serrinoka}. This was proved by Farina and Roncoroni \cite{farina2022} using the Bochner identity (see also \cite{ciraolo}), and its link with a Pohozaev-type identity is the cornerstone of the rigidity phenomenon in this setting.

More precisely, the following holds.

\begin{lemma}[\cite{farina2022}]\label{farina-roncoroni}
Let $(M^{n},g)$ be an $n$-dimensional Riemannian manifold satisfying
\begin{equation}\label{ricci_bound}
\operatorname{Ric} \geq (n-1)k g, \quad k \in \mathbb{R}.
\end{equation}
Let $\Omega \subset M$ be a domain and $u \in C^{2}(\Omega)$ a solution of
\[
\Delta u + nku = -1 \quad \text{in } \Omega.
\]
Then the $P$-function defined in \eqref{pfunction} satisfies
\begin{equation}\label{new:DeltaP}
\Delta P = 2\Bigl(|\mathring{\nabla}^{2}u|^{2}+\bigl(\operatorname{Ric}-(n-1)kg\bigr)(\nabla u,\nabla u)\Bigr)\geq 0 .
\end{equation}
Moreover, $\Delta P\equiv 0$ in $\Omega$ if and only if
\begin{equation}\label{new:hess}
\nabla^2 u = -\left(\frac{1}{n} + ku\right)g \quad\text{and}\quad \operatorname{Ric}(\nabla u, \nabla u) = (n-1)k|\nabla u|^2 \quad \text{in } \Omega .
\end{equation}
\end{lemma}

The first condition in \eqref{new:hess} is rigid. The next Obata-type result, recently obtained by Accornero and Ciraolo \cite{accornero}, shows that, together with the Dirichlet condition, it determines not only the shape of $\Omega$ and the function $u$, as in \cite[Lemma 6]{farina2022}, but also the metric of $\Omega$. It is the key tool for the equality cases in Theorems~\ref{thm:main} and \ref{thm:gradient_estimate}; we stress that it requires neither a curvature assumption nor completeness of $M$.

\begin{lemma}[{\cite[Lemma 2.3]{accornero}}]\label{new:obata}
Let $(M^{n},g)$ be a Riemannian manifold and let $\Omega\subset M$ be a domain with compact closure and nonempty boundary. Assume that there exists $u\in C^{0}(\overline{\Omega})\cap C^{2}(\Omega)$ such that, for some $k\in\mathbb{R}$,
\[
\nabla^{2}u=-\Bigl(\frac1n+ku\Bigr)g\ \text{ in }\Omega,\qquad u>0\ \text{ in }\Omega,\qquad u=0\ \text{ on }\partial\Omega .
\]
Then $\Omega$ is a metric ball $B_{R}(p)$, it is isometric to a geodesic ball of radius $R$ in the simply connected space form of constant curvature $k$ (with $R<\frac{\pi}{2\sqrt{k}}$ if $k>0$), and $u$ is the radial function
\begin{equation}\label{new:uexpl}
u(r)=
\begin{cases}
\dfrac{1}{nk}\left(\dfrac{\cos(\sqrt{k}\,r)}{\cos(\sqrt{k}\,R)}-1\right), & k>0,\\[10pt]
\dfrac{R^{2}-r^{2}}{2n}, & k=0,\\[10pt]
\dfrac{1}{nk}\left(\dfrac{\cosh(\sqrt{-k}\,r)}{\cosh(\sqrt{-k}\,R)}-1\right), & k<0,
\end{cases}
\qquad r=d_{g}(\cdot,p).
\end{equation}
\end{lemma}

\begin{remark}\label{new:remconv}
Conversely, on a geodesic ball of radius $R$ in the space form of curvature $k$ (with $R<\frac{\pi}{2\sqrt{k}}$ if $k>0$), the function \eqref{new:uexpl} is the positive solution of \eqref{serrinoka} and satisfies \eqref{new:hess}. In particular, on $\partial\Omega$ one has $u_{\nu\nu}=-\frac1n$ and $u_{\nu}$ is constant.
\end{remark}

In Section~\ref{sec_deficit}, we shift our perspective to study the spherical deficit for solutions of \eqref{serrinoka}. This requires a richer geometric setting: Riemannian manifolds endowed with a conformal vector field. Recall that a vector field $X$ on $(M^n,g)$ is called conformal if its flow consists of conformal transformations. Equivalently, there exists a smooth function $V$ on $M$ such that the Lie derivative of the metric satisfies $\mathcal{L}_X g = 2V g$. The function $V$ is called the conformal factor of $X$. When $X$ is also closed (i.e., its dual $1$-form is closed), we obtain a particularly tractable structure that generalizes the position vector field in Euclidean space.

The following lemma, which is treated for example in \cite[Lemma 2.2]{andrade2025}, relates the conformal factor $V$ to the scalar curvature $R$ of the manifold. It will be crucial in the derivation of our general integral identity (Proposition~\ref{intlemma}).

\begin{lemma}\label{lemma1}
Let $(M^{n},g)$ be a Riemannian manifold endowed with a conformal vector field $X$ with conformal factor $V$. Then
\begin{equation}\label{Lapl_V}
-(n-1)\Delta V
=
\frac{1}{2}X(R)+VR,
\end{equation}
where $R$ denotes the scalar curvature of $(M^{n},g)$.
\end{lemma}

We also recall the following result, particularly applicable in the case of Einstein manifolds. It is a consequence of \cite[Lemma 2.1]{kr}, see also \cite[p.~160]{yano}.

\begin{lemma}\label{lemma_tracefree}
Let $(M^{n},g)$, $n\geq 3$, be a Riemannian manifold endowed with a conformal vector field $X$ with conformal factor $V$. Then
\[
\mathcal{L}_{X}\mathring{\operatorname{Ric}}=0
\quad\Longleftrightarrow\quad
\mathring{\nabla}^{2}V=0.
\]

In particular, if $(M^{n},g)$ is Einstein, then
\[
\mathring{\nabla}^{2}V=0.
\]
\end{lemma}

\begin{remark}
We recall that a conformal vector field $X$ is said to be \emph{closed conformal} if the $1$-form dual to $X$ is closed. Equivalently,
\[
\nabla_YX = VY,
\]
for every $Y\in\mathfrak{X}(M)$, where $V$ denotes the conformal factor of $X$. A remarkable property of manifolds admitting a nontrivial closed conformal vector field is that they are locally isometric to a warped product with a one-dimensional factor; see, for instance, \cite[Section 3]{montiel}. This class includes space forms and many other geometrically relevant examples.
\end{remark}

\section{Estimates of the Torsional Rigidity}\label{sec_estimates}

We start this section giving our first result. Such a result may be compared with \cite[Theorem 1.6]{wang2020}.

\begin{theorem}\label{thm:main}
Let $(M^n,g)$ be an $n$-dimensional Riemannian manifold whose Ricci curvature satisfies $\operatorname{Ric}\geq (n-1)kg$, $k\in\mathbb{R}$, let $\Omega\subset M$ be a domain with compact closure and $C^{2}$ boundary, and let $u$ be a solution of
\begin{equation}\label{eq:main_problem}
\Delta u + nku = -1\quad \text{in } \Omega,\qquad u = 0\quad \text{on } \partial\Omega,
\end{equation}
with $u>0$ in $\Omega$. Let $\nu$ denote the outward unit normal vector field along $\partial\Omega$.

\begin{itemize}
\item[(a)] We have
\begin{equation}\label{eq:min_unn}
\min_{x\in\partial\Omega} u_{\nu\nu}(x)\leq -\frac{1}{n}.
\end{equation}

\item[(b)] Let $\operatorname{Area}(\partial\Omega)$, $H$, $T(\Omega)$, and $\operatorname{Vol}(\Omega)$ denote the area of $\partial\Omega$, the mean curvature of $\partial\Omega$ (the trace of the second fundamental form with respect to $\nu$), the torsional rigidity of $\Omega$, and the volume of $\Omega$, respectively. If $H\geq 0$ on $\partial\Omega$, then
\begin{equation}\label{eq:integral_unn}
\int_{\partial\Omega}u_{\nu\nu}d\sigma\leq \left(\frac{n-1}{n}\right)^{1/2}\left(\operatorname{Vol}(\Omega) + nkT(\Omega)\right)^{1/2}\left(\int_{\partial\Omega}Hd\sigma\right)^{1/2} - \operatorname{Area}(\partial\Omega).
\end{equation}
\end{itemize}
Moreover, equality holds in \eqref{eq:min_unn}, or in \eqref{eq:integral_unn}, if and only if $\Omega$ is isometric to a geodesic ball in the simply connected space form of constant curvature $k$. In this case, $u$ is the radial function \eqref{new:uexpl}.
\end{theorem}

\begin{proof}

By Lemma \ref{new:lempos}, we have $\Delta u<0$ in $\Omega$ and
\begin{equation}\label{eq:unu_negative}
u_{\nu}(x)<0,\qquad \forall x\in\partial\Omega.
\end{equation}

Using the classical Bochner formula
\[\frac{1}{2}\Delta |\nabla u|^2 = |\nabla^2 u|^2 +\langle \nabla u,\nabla \Delta u\rangle +\operatorname{Ric}(\nabla u,\nabla u),\]
together with \eqref{eq:main_problem} and the assumption $\operatorname{Ric}\geq (n-1)kg$, we obtain
\begin{equation}\label{eq:bochner_estimate}
\frac{1}{2}\Delta |\nabla u|^2\geq |\nabla^2 u|^2 -nk|\nabla u|^2 +(n-1)k|\nabla u|^2 = |\nabla^2 u|^2 -k|\nabla u|^2.
\end{equation}
Integrating over $\Omega$ and applying the divergence theorem, we get
\begin{equation}\label{eq:integral_bochner}
\int_{\Omega}|\nabla^2 u|^2dv -k\int_{\Omega}|\nabla u|^2dv\leq \frac{1}{2}\int_{\partial\Omega}\frac{\partial}{\partial\nu} |\nabla u|^2d\sigma.
\end{equation}
Since $u = 0$ on $\partial\Omega$, it follows that
\begin{equation}\label{eq:boundary_grad}
\nabla u = u_{\nu}\nu\ \text{on }\partial\Omega,\quad\text{and}\quad \frac{\partial}{\partial\nu} |\nabla u|^2 = 2\nabla^2 u(\nabla u,\nu) = 2u_{\nu}u_{\nu\nu}\ \text{on }\partial\Omega.
\end{equation}
Combining this with \eqref{eq:integral_bochner}, we conclude that
\begin{equation}\label{eq:inequality_hessian}
\int_{\Omega}|\nabla^2 u|^2dv -k\int_{\Omega}|\nabla u|^2dv\leq \int_{\partial\Omega}u_{\nu}u_{\nu\nu}d\sigma.
\end{equation}

Define
\[L = \min_{x\in\partial\Omega} u_{\nu\nu}(x).\]
Then, by \eqref{eq:unu_negative} and the divergence theorem, we obtain
\[\int_{\partial\Omega}u_{\nu}u_{\nu\nu}d\sigma\leq L\int_{\partial\Omega}u_{\nu}d\sigma = L\int_{\Omega}\Delta u dv= -L\int_{\Omega}(nku+1)dv.\]
Substituting this into \eqref{eq:inequality_hessian}, we infer
\begin{equation}\label{eq:after_L}
\int_{\Omega}|\nabla^2 u|^2dv -k\int_{\Omega}|\nabla u|^2dv\leq -nkL\int_{\Omega}u \, dv- L\operatorname{Vol}(\Omega).
\end{equation}
Using the definition of torsional rigidity $T(\Omega)=\int_{\Omega}udv$ and the identity $\int_{\Omega}|\nabla u|^2dv = T(\Omega)+nk\int_{\Omega}u^2dv$, we rewrite \eqref{eq:after_L} as
\begin{equation}\label{eq:before_CS}
\int_{\Omega}|\nabla^2 u|^2dv -nk^2\int_{\Omega}u^2dv +kT(\Omega)(-1+nL)\leq -L\operatorname{Vol}(\Omega).
\end{equation}

On the other hand, the Cauchy-Schwarz inequality together with \eqref{eq:main_problem} yields
\[|\nabla^2 u|^2\geq \frac{(\Delta u)^2}{n} = \frac{(1+nku)^2}{n} = nk^2 u^2 +2ku +\frac{1}{n}.\]
Integrating over $\Omega$, we obtain
\begin{equation}\label{eq:CS_integral}
\int_{\Omega}|\nabla^2 u|^2dv\geq nk^2\int_{\Omega}u^2\, dv +2k\int_{\Omega}u\, dv +\frac{1}{n}\operatorname{Vol}(\Omega).
\end{equation}
Substituting \eqref{eq:CS_integral} into \eqref{eq:before_CS}, we arrive at
\[L\left(nkT(\Omega) + \operatorname{Vol}(\Omega)\right)\leq -\frac{1}{n}\left(nkT(\Omega) + \operatorname{Vol}(\Omega)\right).\]
Since $nkT(\Omega) + \operatorname{Vol}(\Omega) > 0$ by Lemma \ref{new:lempos}, we conclude that
\begin{equation}\label{eq:L_bound}
L\leq -\frac{1}{n}.
\end{equation}
This proves the inequality in part (a) of the theorem.

We now analyze the equality case. If equality holds in \eqref{eq:L_bound}, then equality holds in \eqref{eq:bochner_estimate} and in \eqref{eq:CS_integral}, hence
\begin{equation}\label{eq:rigidity_conditions}
\nabla^2 u = \frac{\Delta u}{n}\, g=-\Bigl(\frac1n+ku\Bigr)g\qquad\text{and}\qquad \operatorname{Ric}(\nabla u,\nabla u) = (n-1)k|\nabla u|^2\qquad\text{in }\Omega.
\end{equation}
Since $u>0$ in $\Omega$ and $u=0$ on $\partial\Omega$, Lemma \ref{new:obata} applies and shows that $\Omega$ is isometric to a geodesic ball in the space form of curvature $k$, with $u$ given by \eqref{new:uexpl}. Conversely, on such a ball $u_{\nu\nu}=-\frac1n$ on $\partial\Omega$ by Remark \ref{new:remconv}, so equality holds in \eqref{eq:min_unn}.

We now prove part (b). Since $u$ vanishes on $\partial\Omega$, it follows from \eqref{eq:main_problem} that $\Delta u = -1$ on $\partial\Omega$. On the other hand, the decomposition of the Laplacian along the boundary gives
\[\Delta u = \Delta_{\partial\Omega}u + Hu_{\nu} + u_{\nu\nu}.\]
Because $u = 0$ on $\partial\Omega$, we have $\Delta_{\partial\Omega}u = 0$. Hence, integrating over $\partial\Omega$, we deduce
\begin{equation}\label{eq:boundary_decomposition}
\operatorname{Area}(\partial\Omega) + \int_{\partial\Omega}u_{\nu\nu}d\sigma = -\int_{\partial\Omega}Hu_{\nu}d\sigma.
\end{equation}

We now invoke Reilly's formula \cite{reilly}, which in the present setting reads
\[\int_{\Omega}\left((\Delta u)^2 - |\nabla^2 u|^2 - \operatorname{Ric}(\nabla u,\nabla u)\right) dv= \int_{\partial\Omega}H u_{\nu}^2d\sigma.\]
Substituting $\Delta u = -(1 + nku)$ into this identity, we obtain
\[\int_{\partial\Omega}H u_{\nu}^2d\sigma = \int_{\Omega}\left((1 + nku)^2 - |\nabla^2 u|^2 - \operatorname{Ric}(\nabla u,\nabla u)\right)dv.\]
Using the estimates $|\nabla^2 u|^2\geq \frac{(\Delta u)^2}{n}$ and $\operatorname{Ric}(\nabla u,\nabla u)\geq (n-1)k|\nabla u|^2$, we infer
\begin{equation}\label{eq:reilly_estimate}
\int_{\partial\Omega}H u_{\nu}^2d\sigma\leq \frac{n-1}{n}\int_{\Omega}\left(1 + 2nku + n^2k^2u^2\right)dv - (n-1)k\int_{\Omega}|\nabla u|^2dv.
\end{equation}
Moreover, equality holds in \eqref{eq:reilly_estimate} if and only if \eqref{eq:rigidity_conditions} is satisfied.

Integration by parts together with the boundary condition $u = 0$ on $\partial\Omega$ gives
\[\int_{\Omega}|\nabla u|^2dv = -\int_{\Omega}u\Delta udv = \int_{\Omega}udv + nk\int_{\Omega}u^2dv = T(\Omega) + nk\int_{\Omega}u^2dv.\]
Combining this with \eqref{eq:reilly_estimate}, we arrive at
\begin{equation}\label{eq:H_bound}
\int_{\partial\Omega}H u_{\nu}^2d\sigma\leq \frac{n-1}{n}\operatorname{Vol}(\Omega) + (n-1)kT(\Omega).
\end{equation}
Since $H\geq 0$ on $\partial\Omega$ and $u_{\nu}<0$ by Hopf's lemma, H\"older's inequality yields
\[-\int_{\partial\Omega}Hu_{\nu}d\sigma\leq \left(\int_{\partial\Omega}H u_{\nu}^2d\sigma\right)^{1/2}\left(\int_{\partial\Omega}Hd\sigma\right)^{1/2}.\]
Substituting \eqref{eq:H_bound} into this estimate and then into \eqref{eq:boundary_decomposition}, we conclude
\[\int_{\partial\Omega}u_{\nu\nu} d\sigma\leq \left(\frac{n-1}{n}\right)^{1/2}\left(\operatorname{Vol}(\Omega) + nkT(\Omega)\right)^{1/2}\left(\int_{\partial\Omega}Hd\sigma\right)^{1/2} - \operatorname{Area}(\partial\Omega),\]
which is exactly \eqref{eq:integral_unn}. If equality holds in \eqref{eq:integral_unn}, then equality holds in \eqref{eq:reilly_estimate}, so \eqref{eq:rigidity_conditions} is satisfied and the conclusion follows from Lemma \ref{new:obata} as in part (a). Conversely, on a geodesic ball of a space form of curvature $k$, the function \eqref{new:uexpl} satisfies \eqref{eq:rigidity_conditions}, so equality holds in \eqref{eq:reilly_estimate}; moreover $u_{\nu}$ and $H$ are constant on $\partial\Omega$, so equality also holds in H\"older's inequality, and therefore in \eqref{eq:integral_unn}.
This completes the proof of Theorem \ref{thm:main}.
\end{proof}


Now, we explicitly give a limitation for the torsional rigidity in manifolds with Ricci bounded below. Such a result could be compared with \cite[Theorem 1.7]{wang2020}.

\begin{theorem} (Theorem \ref{MainTh_1})\label{thm:gradient_estimate}
Let $(M^n,g)$ be an $n$-dimensional Riemannian manifold whose Ricci curvature satisfies $\operatorname{Ric}\geq (n-1)kg$, $k\in\mathbb{R}$, let $\Omega\subset M$ be a domain with compact closure and $C^{2}$ boundary, and let $u$ be a solution of \eqref{eq:main_problem} with $u>0$ in $\Omega$. Then
\begin{equation}\label{eq:max_gradient_estimate}
\max_{\partial\Omega}|\nabla u|^2\geq \frac{1}{\operatorname{Vol}(\Omega) + nkT(\Omega)}\cdot\frac{n+2}{n}\left[T(\Omega) + nk\int_{\Omega}\left(u^2 + u|\nabla u|^2\right)dv\right].
\end{equation}
Moreover, equality holds in \eqref{eq:max_gradient_estimate} if and only if $\Omega$ is isometric to a geodesic ball in the simply connected space form of constant curvature $k$. In this case, $u$ is the radial function \eqref{new:uexpl}.
\end{theorem}

\begin{proof}
Throughout the proof, we consider the auxiliary $P$-function
\[P = |\nabla u|^2 +\frac{2}{n}u + ku^2.\]
Using integration by parts, together with the boundary conditions $u = 0$ on $\partial\Omega$ and $P = |\nabla u|^2 = u_{\nu}^2$ on $\partial\Omega$, we obtain
\begin{eqnarray}\label{eq323}
\int_{\Omega}(u\Delta P - P\Delta u)dv = \int_{\partial\Omega}(uP_{\nu} - Pu_{\nu})d\sigma = -\int_{\partial\Omega}u_{\nu}^{3}d\sigma.
\end{eqnarray}
Recall that, under the assumption $\operatorname{Ric}\geq (n-1)kg$,
\[\Delta P = 2\Big(|\mathring{\nabla}^2 u|^2 +(\operatorname{Ric} - (n-1)kg)(\nabla u,\nabla u)\Big)\geq 0.\]
Moreover, equality holds if and only if \eqref{eq:rigidity_conditions} is satisfied.

Let
\[m = \max_{\partial\Omega}|\nabla u|.\]
Substituting the above information into \eqref{eq323}, and using the fact that $u_{\nu}<0$ on $\partial\Omega$ and $u>0$ in $\Omega$ (see Lemma \ref{new:lempos}), we deduce
\[-\int_{\Omega}P\Delta udv\leq -m^2\int_{\partial\Omega}u_{\nu}d\sigma = m^2\int_{\Omega}(-\Delta u) dv= m^2\big(\operatorname{Vol}(\Omega) + nkT(\Omega)\big),\]
where we used $-\Delta u = 1 + nku$ and $\int_{\Omega}udv = T(\Omega)$.

Next, using the equation $-\Delta u = 1 + nku$, we compute
\begin{eqnarray*}
-\int_{\Omega}(\Delta u)P dv&=& \int_{\Omega}(1 + nku)\left(|\nabla u|^2 + \frac{2}{n}u + ku^2\right)dv\\
&=&\displaystyle\int_{\Omega}|\nabla u|^2dv + \frac{2}{n}\int_{\Omega}udv + k\int_{\Omega}u^2dv + nk\int_{\Omega}u|\nabla u|^2dv + 2k\int_{\Omega}u^2dv + nk^2\int_{\Omega}u^3dv.
\end{eqnarray*}

Using the identity $\int_{\Omega}|\nabla u|^2dv = \int_{\Omega}udv + nk\int_{\Omega}u^2dv$, it follows that
\[-\int_{\Omega}(\Delta u)P dv= \frac{n+2}{n}\int_{\Omega}udv + k(n+3)\int_{\Omega}u^2dv + nk\int_{\Omega}u|\nabla u|^2dv + nk^2\int_{\Omega}u^3dv.\]

On the other hand, integrating by parts and using $\operatorname{div}(u^2\nabla u) = u^2\Delta u + 2u|\nabla u|^2$, together with the boundary condition $u = 0$ on $\partial\Omega$, we obtain
\[nk^2\int_{\Omega}u^3dv = 2k\int_{\Omega}u|\nabla u|^2dv - k\int_{\Omega}u^2dv.\]

Therefore,
\[-\int_{\Omega}(\Delta u)Pdv = \frac{n+2}{n}\int_{\Omega}udv + k(n+2)\left(\int_{\Omega}u^2dv + \int_{\Omega}u|\nabla u|^2dv\right).\]

Finally, combining the estimates, we conclude that
\[m^2\big(\operatorname{Vol}(\Omega) + nkT(\Omega)\big)\geq \frac{n+2}{n}\left[T(\Omega) + nk\int_{\Omega}\left(u^2 + u|\nabla u|^2\right)dv\right].\]

Since $\operatorname{Vol}(\Omega) + nkT(\Omega) > 0$ by Lemma \ref{new:lempos}, inequality \eqref{eq:max_gradient_estimate} follows. If equality holds, then $\int_{\Omega}u\,\Delta P\,dv=0$; since $u>0$ and $\Delta P\geq 0$ in $\Omega$, we get $\Delta P\equiv0$, hence \eqref{new:hess} by Lemma \ref{farina-roncoroni}, and Lemma \ref{new:obata} shows that $\Omega$ is isometric to a geodesic ball in the space form of curvature $k$. Conversely, on such a ball $\Delta P\equiv 0$ and $|\nabla u|$ is constant on $\partial\Omega$ (Remark \ref{new:remconv}), so every inequality used above is an equality.
\end{proof}

\begin{remark}
In the direction of \cite[Theorem 1.7]{wang2020}, the results of this section apply verbatim when $\Omega$ is replaced by a compact Riemannian manifold with boundary $(M^{n},g)$ satisfying $\operatorname{Ric}\geq(n-1)kg$ and carrying a positive solution of
\begin{equation}\label{Dirichlet3}
\Delta u + n k u = -1 \quad \text{in } M,
\qquad
u = 0 \quad \text{on } \partial M:
\end{equation}
equality in \eqref{eq:min_unn}, \eqref{eq:integral_unn} or \eqref{eq:max_gradient_estimate} holds if and only if $(M^{n},g)$ is isometric to a geodesic ball in the space form of curvature $k$. For $k=0$ this is the rigidity statement of \cite[Theorems 1.6 and 1.7]{wang2020}.
\end{remark}


\section{An Integral identity and the spherical deficit}\label{sec_deficit}

Along this section, we assume that $(M^{n},g)$ is a Riemannian manifold endowed with a conformal vector field $X$. This means that the Lie derivative of the metric satisfies
\[
\mathcal{L}_{X}g = 2Vg,
\]
for some smooth function $V\colon M\to\mathbb{R}$, called the conformal factor of $X$. Although conformal vector fields naturally arise from one-parameter families of conformal transformations, it is worthwhile to mention some classical examples.

\begin{example}[Space forms]\label{space_forms}
The simply connected space forms admit natural conformal vector fields.

\begin{itemize}
    \item In the Euclidean space $\mathbb{R}^{n}$ endowed with the Euclidean metric $g_{0}$, the position vector field is conformal with conformal factor $V\equiv 1$.

    \item In the hyperbolic space $\mathbb{H}^{n}$ written as
    \[
    \mathbb{H}^{n}=([0,\infty)\times \mathbb{S}^{n-1},\,dr^{2}+\sinh^{2}(r)\,g_{\mathbb{S}^{n-1}}),
    \]
    the vector field \(X=\sinh(r)\,\partial_{r}\)
    is conformal with conformal factor $V=\cosh(r)$.

    \item In the sphere $\mathbb{S}^{n}$ written as
    \[
    \mathbb{S}^{n}=((0,\pi)\times \mathbb{S}^{n-1},\,dr^{2}+\sin^{2}(r)\,g_{\mathbb{S}^{n-1}}),
    \]
    the vector field $X=\sin(r)\,\partial_{r}$
    is conformal with conformal factor $V=\cos(r)$.
\end{itemize}
\end{example}

\begin{example}[Warped products]
More generally, let
\[
(M^{n},g)=\bigl((a,b)\times N^{n-1},\,dr^{2}+h(r)^{2}g_{N}\bigr)
\]
be a warped product manifold, where $(N^{n-1},g_{N})$ is a Riemannian manifold and $h\colon(a,b)\to\mathbb{R}_{+}$ is smooth. Then the vector field $X=h(r)\,\partial_{r}$ is conformal with conformal factor $V=h'(r)$.
\end{example}

Furthermore, we consider a reference solution $\varphi$ of \eqref{Dirichlet2}, namely a function satisfying
\begin{equation}\label{deltaphi1}
\nabla^{2}\varphi
=
-\left(\frac{1}{n}+k\varphi\right)g
\end{equation}
in $\Omega$. Taking the trace of \eqref{deltaphi1}, we obtain
\begin{equation}\label{deltaphi}
\Delta \varphi
=
-\left(1+nk\varphi\right).
\end{equation}

Before stating our first result, we present some examples of such solutions in particular settings.

\begin{example}[Euclidean space]
In the Euclidean case, where $k=0$, a corresponding reference solution is given by
\[
\varphi(x)=\frac{1}{2n}\bigl(a-|x-b|^{2}\bigr),
\]
where $a\in\mathbb{R}$ and $b\in\mathbb{R}^{n}$ are constants.
\end{example}

\begin{example}[Sphere and hyperbolic space]\label{ex:space_forms_phi}
Consider the space forms described in Example~\ref{space_forms}. Namely,
\begin{itemize}
    \item the hyperbolic space, for which
    \[
    h(r)=\sinh r
    \qquad \text{and} \qquad k=-1,
    \]

    \item and the sphere, for which
    \[
    h(r)=\sin r
    \qquad \text{and} \qquad k=1.
    \]
\end{itemize}

In both cases, the warping function satisfies
\[
h''=-kh,
\]
and consequently the radial function
\[
\varphi(r)=c_{0}h'(r)-\frac{1}{nk},
\]
where $c_{0}$ is a constant, satisfies
\[
\nabla^{2}\varphi
=
-\left(\frac{1}{n}+k\varphi\right)g.
\]
\end{example}


In the following, we prove an integral identity which constitutes the core of this section. As discussed in the introduction, it generalizes \cite[Theorem 2.1]{poggesi}, and the strength of identities of this type lies in their ability to measure the spherical deficit through the quantity
\[
\Delta P
=
2\left(
|\mathring{\nabla}^{2}u|^{2}
+
(\operatorname{Ric}-(n-1)kg)(\nabla u,\nabla u)
\right),
\]
which appears on the left-hand side (see Remark~\ref{remark_main} below).
\begin{proposition}[General Integral Identity]\label{intlemma}
Let $(M^n, g)$ be a Riemannian manifold endowed with a conformal vector field $X$ with conformal factor $V$. Let $\Omega$ be a domain with compact closure and $C^2$ boundary, and let $u\in C^{3}(\Omega)\cap C^{2}(\overline{\Omega})$ be a solution of
\[
\Delta u=-(1+nku), \qquad k\in\mathbb{R},
\]
in $\Omega$. Assume further that $\varphi\,\in C^{2}(\overline{\Omega})$ is a solution of \eqref{deltaphi1}. Then the following identity holds:
\begin{eqnarray}\label{eqlema}
\displaystyle\int_{\Omega}-Vu\Delta Pdv&=&\displaystyle\int_{\partial \Omega}|\nabla u|^2[V(u-\varphi)_{\nu}-(u-\varphi)V_{\nu}]d\sigma
+2\displaystyle\int_{\partial \Omega}u\nabla^2u((u-\varphi)\nabla V-V\nabla (u-\varphi),\nu)d\sigma\nonumber\\
&&+k\displaystyle\int_{\partial \Omega}u^2[(u-\varphi)V_{\nu}-V(u-\varphi)_{\nu}]d\sigma
-\dfrac{2}{n(n-1)}\displaystyle\int_{\Omega}Vu(u-\varphi)(R-n(n-1)k)dv\nonumber\\
&&-2\displaystyle\int_{\Omega}u(u-\varphi)(\operatorname{Ric}-(n-1)kg)(\nabla u,\nabla V)dv
-2\displaystyle\int_{\Omega}Vu(\operatorname{Ric}-(n-1)kg)(\nabla u,\nabla \varphi)dv\nonumber\\
&&-\dfrac{k}{n-1}\displaystyle\int_{\Omega}Vu^2(u-\varphi)(R-(n-1)nk)dv
-\dfrac{1}{n-1}\displaystyle\int_{\Omega}V|\nabla u|^2(u-\varphi)(R-(n-1)nk)dv\nonumber\\
&&-\dfrac{k}{2(n-1)}\displaystyle\int_{\Omega}u^2(u-\varphi)X(R)dv
-\dfrac{1}{2(n-1)}\displaystyle\int_{\Omega}X(R)|\nabla u|^2(u-\varphi)dv\nonumber\\
&&-\dfrac{1}{n(n-1)}\displaystyle\int_{\Omega}u(u-\varphi)X(R)dv
-2\displaystyle\int_{\Omega}u(u-\varphi)\langle\nabla^2u,\mathring{\nabla}^{2}V\rangle dv.
\end{eqnarray}
\end{proposition}

\begin{proof}
The proof is a long calculation.
Initially, observe that
\[
\operatorname{div}\bigl(Vu\nabla P-P\nabla(Vu)\bigr)
=
Vu\,\Delta P-P\,\Delta(Vu).
\]
Integrating this identity over \(\Omega\) and applying the divergence theorem, we obtain
\begin{equation}\label{eqin}
-\int_{\Omega}Vu\,\Delta Pdv
=
-\int_{\partial\Omega}Vu\,P_{\nu}d\sigma
+\int_{\partial\Omega}P\,(Vu)_{\nu}d\sigma
-\int_{\Omega}P\,\Delta(Vu)dv.
\end{equation}

Next, using the equation \(\Delta u=-(1+nku)\) together with Lemma \ref{lemma1}, we compute
\begin{align}\label{eqnablaVu}
\Delta(Vu)
&=
V\Delta u+u\Delta V+2\langle\nabla V,\nabla u\rangle \nonumber\\
&=
-V(1+nku)-\frac{uX(R)}{2(n-1)}-\frac{R}{n-1}Vu
+2\langle\nabla V,\nabla u\rangle \nonumber\\
&=
-V-\frac{Vu}{n-1}\bigl(R+n(n-1)k\bigr)
+2\langle\nabla V,\nabla u\rangle
-\frac{uX(R)}{2(n-1)}.
\end{align}

Substituting \eqref{eqnablaVu} into \eqref{eqin}, we arrive at
\begin{align}\label{eq43}
-\int_{\Omega}Vu\,\Delta Pdv
={}&
-\int_{\partial\Omega}Vu\,P_{\nu}d\sigma
+\int_{\partial\Omega}P\,(Vu)_{\nu}d\sigma
+\int_{\Omega}PVdv \nonumber\\
&\quad
+\frac{1}{n-1}\int_{\Omega}PVu\bigl(R+n(n-1)k\bigr)dv
-2\int_{\Omega}P\langle \nabla V,\nabla u\rangle dv\nonumber\\
&\quad
+\frac{1}{2(n-1)}\int_{\Omega}Pu\,X(R)dv.
\end{align}

We now make use of the auxiliary function \(\varphi\). Applying integration by parts, we obtain
\begin{align}\label{eq44}
\int_{\partial\Omega}P(Vu)_{\nu}d\sigma
&=
\int_{\partial\Omega}P\bigl(V(u-\varphi)\bigr)_{\nu}d\sigma
+
\int_{\partial\Omega}P(\varphi V)_{\nu}d\sigma
\nonumber\\
&=
\int_{\partial\Omega}PV(u_{\nu}-\varphi_{\nu})d\sigma
+
\int_{\partial\Omega}Pu\,V_{\nu}d\sigma
+
\int_{\partial\Omega}PV\varphi_{\nu}d\sigma
\nonumber\\
&=
\int_{\partial\Omega}PV(u_{\nu}-\varphi_{\nu})d\sigma
+
\int_{\Omega}\operatorname{div}(uP\nabla V)dv
+
\int_{\partial\Omega}PV\varphi_{\nu}d\sigma
\nonumber\\
&=
\int_{\partial\Omega}PV(u_{\nu}-\varphi_{\nu})d\sigma
+
\int_{\Omega}P\langle\nabla u,\nabla V\rangle dv
+
\int_{\Omega}u\langle\nabla P,\nabla V\rangle dv
\nonumber\\
&\quad
+
\int_{\Omega}uP\,\Delta Vdv
+
\int_{\partial\Omega}PV\varphi_{\nu}d\sigma
\nonumber\\
&=
\int_{\partial\Omega}PV(u_{\nu}-\varphi_{\nu})d\sigma
+
\int_{\Omega}P\langle\nabla u,\nabla V\rangle dv
+
\int_{\Omega}u\langle\nabla P,\nabla V\rangle dv
\nonumber\\
&\quad
-
\frac{1}{n-1}
\int_{\Omega}uP\left(\frac{X(R)}{2}+VR\right)dv
+
\int_{\partial\Omega}PV\varphi_{\nu}d\sigma.
\end{align}

Next, integrating by parts once again, we obtain
\begin{align}\label{eq45}
\int_{\partial\Omega}VP\varphi_{\nu}d\sigma
&=
\int_{\Omega}\operatorname{div}(VP\nabla\varphi)dv
\nonumber\\
&=
\int_{\Omega}VP\,\Delta\varphi dv
+
\int_{\Omega}P\langle\nabla V,\nabla\varphi\rangle dv
+
\int_{\Omega}V\langle\nabla P,\nabla\varphi\rangle dv
\nonumber\\
&=
-\int_{\Omega}VPdv
-nk\int_{\Omega}VP\varphi dv
+
\int_{\Omega}P\langle\nabla V,\nabla\varphi\rangle dv
+
\int_{\Omega}V\langle\nabla P,\nabla\varphi\rangle dv,
\end{align}
where we used the identity \(\Delta\varphi=-(1+nk\varphi)\). On the other hand, recalling that the \(P\)-function is given by
\begin{equation}\label{P_func1}
P=|\nabla u|^{2}+\frac{2}{n}u+ku^{2},
\end{equation}
we compute
\begin{equation}\label{P_funct2}
\nabla P
=
2\nabla^{2}u(\nabla u,\cdot)
+\frac{2}{n}\nabla u
+2ku\nabla u.
\end{equation}
Substituting these expressions into \eqref{eq45}, we arrive at
\begin{align}\label{eq46}
\int_{\partial\Omega}VP\varphi_{\nu} d\sigma
&=
-\int_{\Omega}VP dv
-nk\int_{\Omega}VP\varphi dv
+
\int_{\Omega}P\langle\nabla V,\nabla\varphi\rangle dv
\nonumber\\
&\quad
+
2\int_{\Omega}V\nabla^{2}u(\nabla u,\nabla\varphi)dv
+
\frac{2}{n}\int_{\Omega}V\langle\nabla u,\nabla\varphi\rangle dv
\nonumber\\
&\quad
+
2k\int_{\Omega}uV\langle\nabla u,\nabla\varphi\rangle dv.
\end{align}

In order to handle the Hessian term appearing in \eqref{eq46}, we recall the Ricci identity
\begin{equation}\label{Ricci_id}
\operatorname{div}(\nabla^2u)
=
\nabla(\Delta u)+\operatorname{Ric}(\nabla u,\cdot).
\end{equation}
Using this formula, we compute
\begin{align}\label{eq_auxiliar}
\operatorname{div}\bigl(Vu\nabla_{\nabla\varphi}\nabla u\bigr)
&=
\operatorname{div}\bigl(Vu\,\nabla^{2}u(\nabla\varphi,\cdot)\bigr)
\nonumber\\
&=
u\nabla^{2}u(\nabla V,\nabla\varphi)
+
V\langle\nabla u,\nabla_{\nabla\varphi}\nabla u\rangle
+
Vu\,\operatorname{div}(\nabla_{\nabla\varphi}\nabla u)
\nonumber\\
&=
u\nabla^{2}u(\nabla V,\nabla\varphi)
+
V\nabla^{2}u(\nabla u,\nabla\varphi)
+
Vu\,\operatorname{Ric}(\nabla u,\nabla\varphi)
\nonumber\\
&\quad
+
Vu\langle\nabla\Delta u,\nabla\varphi\rangle
+
Vu\langle\nabla^{2}u,\nabla^{2}\varphi\rangle.
\end{align}

Integrating \eqref{eq_auxiliar} over \(\Omega\) and using the divergence theorem, we deduce
\begin{align}\label{eq47}
\int_{\Omega}V\nabla^{2}u(\nabla u,\nabla\varphi)dv
&=
\int_{\partial\Omega}Vu\,\nabla^{2}u(\nabla\varphi,\nu)d\sigma
-
\int_{\Omega}u\nabla^{2}u(\nabla V,\nabla\varphi)dv
\nonumber\\
&\quad
-
\int_{\Omega}Vu(\operatorname{Ric}-nkg)(\nabla u,\nabla\varphi)dv
-
\int_{\Omega}Vu\langle\nabla^{2}u,\nabla^{2}\varphi\rangle dv,
\end{align}
where we have used \(\nabla\Delta u=-nk\nabla u\).

Since \(\varphi\) and \(u\) satisfy \eqref{deltaphi1} and \eqref{serrinoka}, respectively, we infer that
\begin{align}
\int_{\Omega}Vu\langle\nabla^{2}u,\nabla^{2}\varphi\rangle dv
&=
\int_{\Omega}Vu\left[-\left(\frac{1}{n}+k\varphi\right)\Delta u\right]dv
\nonumber\\
&=
\frac{1}{n}\int_{\Omega}Vu\,\Delta\varphi\,(-1-nku)dv
\nonumber\\
&=
-\frac{1}{n}\int_{\Omega}Vu\,\Delta\varphi dv
-k\int_{\Omega}Vu^{2}\Delta\varphi dv.
\nonumber
\end{align}

Applying the divergence theorem once again, together with the identities
\[
\operatorname{div}(Vu\nabla\varphi)
=
u\langle\nabla V,\nabla\varphi\rangle
+
V\langle\nabla u,\nabla\varphi\rangle
+
Vu\Delta\varphi,
\]
and
\[
\operatorname{div}(Vu^{2}\nabla\varphi)
=
\langle\nabla(Vu^{2}),\nabla\varphi\rangle
+
Vu^{2}\Delta\varphi,
\]
we obtain
\begin{align}
\int_{\Omega}Vu\langle\nabla^{2}u,\nabla^{2}\varphi\rangle dv
&=
\frac{1}{n}\int_{\Omega}u\langle\nabla V,\nabla\varphi\rangle dv
+
\frac{1}{n}\int_{\Omega}V\langle\nabla u,\nabla\varphi\rangle dv
\nonumber\\
&\quad
-
\frac{1}{n}\int_{\partial\Omega}Vu\,\varphi_{\nu}d\sigma
+
k\int_{\Omega}\langle\nabla(Vu^{2}),\nabla\varphi\rangle dv
-
k\int_{\partial\Omega}Vu^{2}\varphi_{\nu}d\sigma.
\nonumber
\end{align}

Proceeding further, we arrive at
\begin{align}\label{eq48}
\int_{\Omega}Vu\langle\nabla^{2}u,\nabla^{2}\varphi\rangle dv
&=
-\frac{1}{n}\int_{\partial\Omega}Vu\,\varphi_{\nu}d\sigma
+
\frac{1}{n}\int_{\Omega}u\langle\nabla V,\nabla\varphi\rangle dv
+
\frac{1}{n}\int_{\Omega}V\langle\nabla u,\nabla\varphi\rangle dv
\nonumber\\
&\quad
+
2k\int_{\Omega}Vu\langle\nabla u,\nabla\varphi\rangle dv
+
k\int_{\Omega}u^{2}\langle\nabla V,\nabla\varphi\rangle dv
-
k\int_{\partial\Omega}Vu^{2}\varphi_{\nu}d\sigma.
\end{align}

Plugging this into \eqref{eq47}, we obtain
\begin{align}\label{eqstar}
\int_{\Omega}V\nabla^{2}u(\nabla u,\nabla\varphi)dv
&=
\int_{\partial\Omega}Vu\,\nabla^{2}u(\nabla\varphi,\nu)d\sigma
-
\int_{\Omega}u\nabla^{2}u(\nabla V,\nabla\varphi)dv
\nonumber\\
&\quad
-
\int_{\Omega}Vu(\operatorname{Ric}-nkg)(\nabla u,\nabla\varphi)dv
+
\frac{1}{n}\int_{\partial\Omega}Vu\,\varphi_{\nu}d\sigma
\nonumber\\
&\quad
-
\frac{1}{n}\int_{\Omega}u\langle\nabla V,\nabla\varphi\rangle dv
-
\frac{1}{n}\int_{\Omega}V\langle\nabla u,\nabla\varphi\rangle dv
\nonumber\\
&\quad
+
k\int_{\partial\Omega}Vu^{2}\varphi_{\nu}d\sigma
-
2k\int_{\Omega}Vu\langle\nabla u,\nabla\varphi\rangle dv
\nonumber\\
&\quad
-
k\int_{\Omega}u^{2}\langle\nabla V,\nabla\varphi\rangle dv.
\end{align}

Substituting \eqref{eqstar} into \eqref{eq46}, we obtain
\begin{align}\label{eqtwostar}
\int_{\partial\Omega}VP\varphi_{\nu}d\sigma
&=
-\int_{\Omega}VPdv
-nk\int_{\Omega}VP\varphi dv
+
\int_{\Omega}P\langle\nabla V,\nabla\varphi\rangle dv
+
\frac{2}{n}\int_{\Omega}V\langle\nabla u,\nabla\varphi\rangle dv
\nonumber\\
&\quad
+
2k\int_{\Omega}uV\langle\nabla u,\nabla\varphi\rangle dv
+
2\int_{\partial\Omega}Vu\nabla^{2}u(\nabla\varphi,\nu) d\sigma
\nonumber\\
&\quad
-
2\int_{\Omega}u\nabla^{2}u(\nabla V,\nabla\varphi)dv
-
2\int_{\Omega}Vu(\operatorname{Ric}-nkg)(\nabla u,\nabla\varphi)dv
\nonumber\\
&\quad
+
\frac{2}{n}\int_{\partial\Omega}Vu\,\varphi_{\nu}d\sigma
-
\frac{2}{n}\int_{\Omega}u\langle\nabla V,\nabla\varphi\rangle dv
-
\frac{2}{n}\int_{\Omega}V\langle\nabla u,\nabla\varphi\rangle dv
\nonumber\\
&\quad
+
2k\int_{\partial\Omega}Vu^{2}\varphi_{\nu}d\sigma
-
4k\int_{\Omega}Vu\langle\nabla u,\nabla\varphi\rangle dv
\nonumber\\
&\quad
-
2k\int_{\Omega}u^{2}\langle\nabla V,\nabla\varphi\rangle dv.
\end{align}

After collecting similar terms, this becomes
\begin{align}
\int_{\partial\Omega}VP\varphi_{\nu} d\sigma
&=
-\int_{\Omega}VP dv
+
\int_{\Omega}P\langle\nabla V,\nabla\varphi\rangle dv
+
2\int_{\partial\Omega}Vu\nabla^{2}u(\nabla\varphi,\nu) d\sigma
\nonumber\\
&\quad
-
2k\int_{\Omega}u^{2}\langle\nabla V,\nabla\varphi\rangle dv
-
2\int_{\Omega}Vu(\operatorname{Ric}-(n-1)kg)(\nabla u,\nabla\varphi) dv
\nonumber\\
&\quad
-
2\int_{\Omega}u\nabla^{2}u(\nabla V,\nabla\varphi) dv
-
\frac{2}{n}\int_{\Omega}u\langle\nabla V,\nabla\varphi\rangle dv
\nonumber\\
&\quad
+
\frac{2}{n}\int_{\partial\Omega}Vu\,\varphi_{\nu} d\sigma
+
2k\int_{\partial\Omega}Vu^{2}\varphi_{\nu} d\sigma
-
nk\int_{\Omega}VP\varphi dv.
\label{eqtwostar_simplified}
\end{align}

Substituting \eqref{eqtwostar_simplified} into \eqref{eq44}, we obtain
\begin{align}\label{eqstar44}
\int_{\partial\Omega}P(Vu)_{\nu}d\sigma
&=
\int_{\partial\Omega}PV(u_{\nu}-\varphi_{\nu})d\sigma
+
\int_{\Omega}P\langle\nabla u,\nabla V\rangle dv
+
\int_{\Omega}u\langle\nabla P,\nabla V\rangle dv
\nonumber\\
&\quad
-
\frac{1}{n-1}
\int_{\Omega}uP\left(\frac{1}{2}X(R)+VR\right)dv
-
\int_{\Omega}VP dv
\nonumber\\
&\quad
+
\int_{\Omega}P\langle\nabla V,\nabla\varphi\rangle dv
+
2\int_{\partial\Omega}Vu\nabla^{2}u(\nabla\varphi,\nu) d\sigma
\nonumber\\
&\quad
-
2k\int_{\Omega}u^{2}\langle\nabla V,\nabla\varphi\rangle dv
-
2\int_{\Omega}Vu(\operatorname{Ric}-(n-1)kg)(\nabla u,\nabla\varphi) dv
\nonumber\\
&\quad
-
2\int_{\Omega}u\nabla^{2}u(\nabla V,\nabla\varphi)dv
-
\frac{2}{n}\int_{\Omega}u\langle\nabla V,\nabla\varphi\rangle dv
\nonumber\\
&\quad
+
\frac{2}{n}\int_{\partial\Omega}Vu\,\varphi_{\nu}d\sigma
+
2k\int_{\partial\Omega}Vu^{2}\varphi_{\nu}d\sigma
-
nk\int_{\Omega}VP\varphi dv.
\end{align}

Returning now to \eqref{eq43}, we obtain
\begin{align}\label{eq49}
-\int_{\Omega}Vu\,\Delta P dv
&=
-\int_{\partial\Omega}Vu\,P_{\nu}d\sigma
+
\int_{\partial\Omega}VP(u_{\nu}-\varphi_{\nu})d\sigma
-
\int_{\Omega}P\langle\nabla V,\nabla(u-\varphi)\rangle dv
\nonumber\\
&\quad
+
\int_{\Omega}u\langle\nabla V,\nabla P\rangle dv
+
2\int_{\partial\Omega}Vu\nabla^{2}u(\nabla\varphi,\nu) d\sigma
\nonumber\\
&\quad
-
2\int_{\Omega}u\nabla^{2}u(\nabla V,\nabla\varphi) dv
-
\frac{2}{n}\int_{\Omega}u\langle\nabla V,\nabla\varphi\rangle dv
\nonumber\\
&\quad
+
\frac{2}{n}\int_{\partial\Omega}Vu\,\varphi_{\nu} d\sigma
+
2k\int_{\partial\Omega}Vu^{2}\varphi_{\nu} d\sigma
\nonumber\\
&\quad
-
2k\int_{\Omega}u^{2}\langle\nabla V,\nabla\varphi\rangle dv
+
nk\int_{\Omega}VP(u-\varphi) dv
\nonumber\\
&\quad
-
2\int_{\Omega}Vu(\operatorname{Ric}-(n-1)kg)(\nabla u,\nabla\varphi) dv.
\end{align}

Next, we collect the terms in \eqref{eq49} involving integrals over \(\Omega\) that do not contain \(k\). Denoting their sum by
\[
C
=
-\int_{\Omega}P\langle\nabla V,\nabla(u-\varphi)\rangle dv
+
\int_{\Omega}u\langle\nabla V,\nabla P\rangle dv
-
2\int_{\Omega}u\nabla^{2}u(\nabla V,\nabla\varphi)dv
-
\frac{2}{n}\int_{\Omega}u\langle\nabla V,\nabla\varphi\rangle dv,
\]
and expanding the expression of \(P\), see \eqref{P_func1} and \eqref{P_funct2}, we find
\begin{align}\label{eq410}
C
&=
-\int_{\Omega}|\nabla u|^{2}
\langle\nabla V,\nabla(u-\varphi)\rangle dv
+
2\int_{\Omega}u\nabla^{2}u\bigl(\nabla V,\nabla(u-\varphi)\bigr) dv
\nonumber\\
&\quad
+
k\int_{\Omega}u^{2}\langle\nabla V,\nabla u\rangle dv
+
k\int_{\Omega}u^{2}\langle\nabla V,\nabla\varphi\rangle dv.
\end{align}

We now compute the first term appearing in \eqref{eq410}. To begin with, observe that
\[
\operatorname{div}\bigl(-|\nabla u|^{2}(u-\varphi)\nabla V\bigr)
=
-(u-\varphi)\langle\nabla |\nabla u|^{2},\nabla V\rangle
-
|\nabla u|^{2}\langle\nabla(u-\varphi),\nabla V\rangle
-
|\nabla u|^{2}(u-\varphi)\Delta V.
\]

Integrating over \(\Omega\) and applying the divergence theorem, we obtain
\begin{align}\label{eqaux0}
-\int_{\Omega}|\nabla u|^{2}\langle\nabla V,\nabla(u-\varphi)\rangle dv
&=
-\int_{\partial\Omega}|\nabla u|^{2}(u-\varphi)V_{\nu} d\sigma
+
\int_{\Omega}(u-\varphi)\langle\nabla |\nabla u|^{2},\nabla V\rangle dv
\nonumber\\
&\quad
+
\int_{\Omega}|\nabla u|^{2}(u-\varphi)\Delta V dv.
\end{align}
Using the identity \eqref{Lapl_V}, we conclude that
\begin{align}\label{eqaux0_final}
-\int_{\Omega}|\nabla u|^{2}\langle\nabla V,\nabla(u-\varphi)\rangle dv
&=
-\int_{\partial\Omega}|\nabla u|^{2}(u-\varphi)V_{\nu} d\sigma
+
2\int_{\Omega}(u-\varphi)\nabla^2 u(\nabla u,\nabla V) dv
\nonumber\\
&\quad
-
\frac{1}{2(n-1)}
\int_{\Omega}X(R)|\nabla u|^{2}(u-\varphi)dv
\nonumber\\
&\quad
-
\frac{1}{n-1}
\int_{\Omega}|\nabla u|^{2}(u-\varphi)VR dv.
\end{align}

On the other hand, using once again the Ricci identity \eqref{Ricci_id}, we obtain
\begin{align*}
\operatorname{div}\bigl(u(u-\varphi)\nabla^{2}u(\nabla V)\bigr)
&=
(u-\varphi)\nabla^{2}u(\nabla V,\nabla u)
+
u\nabla^{2}u(\nabla V,\nabla(u-\varphi))
\nonumber\\
&\quad
+
u(u-\varphi)(\operatorname{Ric}-nkg)(\nabla V,\nabla u)
\nonumber\\
&\quad
+
u(u-\varphi)\langle\nabla^{2}u,\nabla^{2}V\rangle.
\end{align*}

Integrating over \(\Omega\) and applying the divergence theorem, we deduce
\begin{align}\label{auxeq1}
\int_{\Omega}u\nabla^{2}u(\nabla V,\nabla(u-\varphi)) dv
&=
\int_{\partial\Omega}u(u-\varphi)\nabla^{2}u(\nabla V,\nu)d\sigma
-
\int_{\Omega}(u-\varphi)\nabla^{2}u(\nabla V,\nabla u)dv
\nonumber\\
&\quad
-
\int_{\Omega}u(u-\varphi)(\operatorname{Ric}-nkg)(\nabla V,\nabla u)dv
-
\int_{\Omega}u(u-\varphi)\langle\nabla^{2}u,\nabla^{2}V\rangle dv.
\end{align}

Since
\[
\nabla^{2}V
=
\mathring{\nabla^{2}V}
+
\frac{\Delta V}{n}\,g,
\]
using \eqref{Lapl_V} we infer that
\begin{align}\label{eqaux1}
\int_{\Omega}u\nabla^{2}u(\nabla V,\nabla(u-\varphi))dv
&=
\int_{\partial\Omega}u(u-\varphi)\nabla^{2}u(\nabla V,\nu)d\sigma
-
\int_{\Omega}(u-\varphi)\nabla^{2}u(\nabla V,\nabla u) dv
\nonumber\\
&\quad
-
\int_{\Omega}u(u-\varphi)(\operatorname{Ric}-nkg)(\nabla V,\nabla u) dv
-
\int_{\Omega}u(u-\varphi)
\langle\nabla^{2}u,\mathring{\nabla^{2}V}\rangle dv
\nonumber\\
&\quad
-
\frac{1}{2n(n-1)}
\int_{\Omega}u(u-\varphi)X(R) dv
-
\frac{1}{n(n-1)}
\int_{\Omega}u(u-\varphi)VR dv
\nonumber\\
&\quad
-
\frac{k}{2(n-1)}
\int_{\Omega}u^{2}(u-\varphi)X(R) dv
-
\frac{k}{n-1}
\int_{\Omega}u^{2}(u-\varphi)VR dv.
\end{align}

Combining \eqref{eq410}, \eqref{eqaux0_final} and \eqref{eqaux1}, we obtain
\begin{align*}
C
&=
-\int_{\partial\Omega}|\nabla u|^{2}(u-\varphi)V_{\nu}d\sigma
+2\int_{\partial\Omega}u(u-\varphi)\nabla^{2}u(\nabla V,\nu)d\sigma
\nonumber\\
&\quad
-2\int_{\Omega}u(u-\varphi)(\operatorname{Ric}-nkg)(\nabla u,\nabla V)dv
-2\int_{\Omega}u(u-\varphi)\langle\nabla^{2}u,\mathring{\nabla^{2}V}\rangle dv
\nonumber\\
&\quad
-\frac{1}{n-1}\int_{\Omega}(u-\varphi)
\left(
\frac{2}{n}uVR
+2kVu^{2}R
+|\nabla u|^{2}VR
\right)dv
\nonumber\\
&\quad
-\frac{1}{n-1}\int_{\Omega}(u-\varphi)
\left(
\frac12 X(R)|\nabla u|^{2}
+\frac{1}{n}uX(R)
+ku^{2}X(R)
\right)dv
+k\int_{\Omega}u^{2}\langle\nabla V,\nabla(u+\varphi)\rangle dv.
\end{align*}

Substituting this expression for \(C\) into \eqref{eq49}, we deduce
\begin{align}\label{eq411}
-\int_{\Omega}Vu\,\Delta P dv
&=
-\int_{\partial\Omega}|\nabla u|^{2}(u-\varphi)V_{\nu} d\sigma
+2\int_{\partial\Omega}u(u-\varphi)\nabla^{2}u(\nabla V,\nu)d\sigma
\nonumber\\
&\quad
-\int_{\partial\Omega}Vu\,P_{\nu}d\sigma
+\int_{\partial\Omega}VP(u_{\nu}-\varphi_{\nu})d\sigma
+2\int_{\partial\Omega}Vu\nabla^{2}u(\nabla\varphi,\nu)d\sigma
\nonumber\\
&\quad
-\frac{2}{n(n-1)}
\int_{\Omega}u(u-\varphi)VR dv
-\frac{2k}{n-1}
\int_{\Omega}Vu^{2}(u-\varphi)R dv
\nonumber\\
&\quad
-2\int_{\Omega}u(u-\varphi)(\operatorname{Ric}-nkg)(\nabla u,\nabla V)dv
+k\int_{\Omega}u^{2}\langle\nabla V,\nabla(u-\varphi)\rangle dv
\nonumber\\
&\quad
-\frac{1}{n-1}
\int_{\Omega}|\nabla u|^{2}(u-\varphi)VR dv
-\frac{1}{2(n-1)}
\int_{\Omega}X(R)|\nabla u|^{2}(u-\varphi)dv
\nonumber\\
&\quad
-\frac{1}{n(n-1)}
\int_{\Omega}u(u-\varphi)X(R)dv
-\frac{k}{n-1}
\int_{\Omega}u^{2}(u-\varphi)X(R) dv
\nonumber\\
&\quad
-2\int_{\Omega}u(u-\varphi)
\langle\nabla^{2}u,\mathring{\nabla^{2}V}\rangle dv
+\frac{2}{n}\int_{\partial\Omega}Vu\,\varphi_{\nu}d\sigma
+2k\int_{\partial\Omega}Vu^{2}\varphi_{\nu}d\sigma
\nonumber\\
&\quad
+nk\int_{\Omega}VP(u-\varphi)dv
-2\int_{\Omega}Vu(\operatorname{Ric}-(n-1)kg)(\nabla u,\nabla\varphi)dv.
\end{align}

We now expand \(P\) and \(P_{\nu}\) in the previous identity to obtain
\begin{align}\label{eq412}
-\int_{\Omega}Vu\,\Delta P dv
&=
-\int_{\partial\Omega}|\nabla u|^{2}(u-\varphi)V_{\nu} d\sigma
+2\int_{\partial\Omega}u(u-\varphi)\nabla^{2}u(\nabla V,\nu) d\sigma
\nonumber\\
&\quad
+\int_{\partial\Omega}V|\nabla u|^{2}(u_{\nu}-\varphi_{\nu}) d\sigma
-2\int_{\partial\Omega}Vu\nabla^{2}u(\nabla(u-\varphi),\nu) d\sigma
\nonumber\\
&\quad
-k\int_{\partial\Omega}Vu^{2}(u-\varphi)_{\nu} d\sigma
-\frac{2}{n(n-1)}
\int_{\Omega}Vu(u-\varphi)(R-n(n-1)k)dv
\nonumber\\
&\quad
+k\int_{\Omega}u^{2}\langle\nabla V,\nabla(u-\varphi)\rangle dv
-2\int_{\Omega}u(u-\varphi)(\operatorname{Ric}-(n-1)kg)(\nabla u,\nabla V)dv
\nonumber\\
&\quad
-2\int_{\Omega}Vu(\operatorname{Ric}-(n-1)kg)(\nabla u,\nabla\varphi) dv
+2k\int_{\Omega}u(u-\varphi)\langle\nabla u,\nabla V\rangle dv
\nonumber\\
&\quad
-\frac{1}{n-1}
\int_{\Omega}V|\nabla u|^{2}(u-\varphi)(R-(n-1)nk) dv
-\frac{2k}{n-1}
\int_{\Omega}Vu^{2}(u-\varphi)
\left(R-\frac{(n-1)nk}{2}\right) dv
\nonumber\\
&\quad
-\frac{1}{2(n-1)}
\int_{\Omega}X(R)|\nabla u|^{2}(u-\varphi)dv
-\frac{1}{n(n-1)}
\int_{\Omega}u(u-\varphi)X(R)dv
\nonumber\\
&\quad
-\frac{k}{n-1}
\int_{\Omega}u^{2}(u-\varphi)X(R)dv
-2\int_{\Omega}u(u-\varphi)
\langle\nabla^{2}u,\mathring{\nabla^{2}V}\rangle dv.
\end{align}

Next, observe that
\[
\operatorname{div}\bigl((u-\varphi)u^{2}\nabla V\bigr)
=
u^{2}\langle\nabla(u-\varphi),\nabla V\rangle
+
(u-\varphi)\langle\nabla u^{2},\nabla V\rangle
+
(u-\varphi)u^{2}\Delta V.
\]
Integrating by parts, we obtain
\begin{align}\label{eqaux4}
k\int_{\Omega}u(u-\varphi)\langle\nabla V,\nabla u\rangle dv
&=
\frac{k}{2}
\int_{\Omega}(u-\varphi)\langle\nabla V,\nabla u^{2}\rangle dv
\nonumber\\
&=
\frac{k}{2}
\Bigg[
\int_{\partial\Omega}(u-\varphi)u^{2}V_{\nu}d\sigma
-
\int_{\Omega}u^{2}\langle\nabla(u-\varphi),\nabla V\rangle dv
-\int_{\Omega}(u-\varphi)u^{2}\Delta V dv
\Bigg]
\nonumber\\
&=
\frac{k}{2}
\Bigg[
\int_{\partial\Omega}(u-\varphi)u^{2}V_{\nu} d\sigma
+
\frac{1}{2(n-1)}
\int_{\Omega}(u-\varphi)u^{2}X(R) dv
\nonumber\\
&\qquad\qquad
+
\frac{1}{n-1}
\int_{\Omega}(u-\varphi)u^{2}VR dv
-
\int_{\Omega}u^{2}\langle\nabla(u-\varphi),\nabla V\rangle dv
\Bigg],
\end{align}
where we used \eqref{Lapl_V}. Therefore,
\begin{align}\label{eq413}
-\int_{\Omega}Vu\,\Delta P dv
&=
\int_{\partial\Omega}|\nabla u|^{2}
\bigl[V(u-\varphi)_{\nu}-(u-\varphi)V_{\nu}\bigr] d\sigma
+
2\int_{\partial\Omega}
u\nabla^{2}u\bigl((u-\varphi)\nabla V-V\nabla(u-\varphi),\nu\bigr)d\sigma
\nonumber\\
&\quad
+
k\int_{\partial\Omega}
u^{2}\bigl[(u-\varphi)V_{\nu}-V(u-\varphi)_{\nu}\bigr]d\sigma
-\frac{2}{n(n-1)}
\int_{\Omega}Vu(u-\varphi)(R-n(n-1)k) dv
\nonumber\\
&\quad
-2\int_{\Omega}u(u-\varphi)
(\operatorname{Ric}-(n-1)kg)(\nabla u,\nabla V)dv
-2\int_{\Omega}Vu
(\operatorname{Ric}-(n-1)kg)(\nabla u,\nabla\varphi)dv
\nonumber\\
&\quad
-\frac{k}{n-1}
\int_{\Omega}Vu^{2}(u-\varphi)(R-(n-1)nk) dv
-\frac{1}{n-1}
\int_{\Omega}V|\nabla u|^{2}(u-\varphi)(R-(n-1)nk)dv
\nonumber\\
&\quad
-\frac{k}{2(n-1)}
\int_{\Omega}u^{2}(u-\varphi)X(R)dv
-\frac{1}{2(n-1)}
\int_{\Omega}X(R)|\nabla u|^{2}(u-\varphi)dv
\nonumber\\
&\quad
-\frac{1}{n(n-1)}
\int_{\Omega}u(u-\varphi)X(R)dv
-2\int_{\Omega}u(u-\varphi)
\langle\nabla^{2}u,\mathring{\nabla^{2}V}\rangle dv.
\end{align}

This completes the proof of \eqref{eqlema}.
\end{proof}

\begin{remark}\label{remark_main}
The identity in Proposition~\ref{intlemma} is particularly relevant to the study of the stability of spherical symmetry in Riemannian manifolds, in contexts related both to Alexandrov's soap bubble theorem and to Serrin's overdetermined problems, through the use of integral identity techniques; see \cite{poggesi1, poggesi, OO21, SZ25}. Identities of this type also play an important role in the analysis of partially overdetermined problems in space forms, where Serrin-type conditions are used to characterize the rigidity of umbilical hypersurfaces (see \cite{guoxia1, guoxia2}). Both applications rely on fundamental integral identities that are of independent interest; see, for instance, \cite[Theorem 2.1]{poggesi} and \cite[Lemma 4.2]{xie2025}.
\end{remark}

In particular, we highlight the form taken by the last result in the special setting of Einstein manifolds.

\begin{proof}[Proof of Theorem \ref{MainTh_2}]
Since $\operatorname{Ric}=(n-1)kg$, we have $R=n(n-1)k$, so that all terms involving
$R-n(n-1)k$, $\operatorname{Ric}-(n-1)kg$ and $X(R)$ in Proposition~\ref{intlemma}
vanish identically. Since $n\geq3$, Lemma~\ref{lemma_tracefree} gives $\mathring{\nabla}^{2}V=0$, and the last term in \eqref{eqlema} vanishes as well. Finally, the second and third boundary integrals in \eqref{eqlema} vanish because $u=0$ on $\partial\Omega$. Therefore, using \eqref{new:DeltaP} and $|\nabla u|^{2}=u_{\nu}^{2}$ on $\partial\Omega$, we obtain
\[
-2\int_{\Omega}Vu\,|\mathring{\nabla}^{2}u|^{2}dv
=
\int_{\partial\Omega}u_{\nu}^{2}
\bigl[V(u-\varphi)_{\nu}-(u-\varphi)V_{\nu}\bigr]d\sigma.
\]

Now, for any constant $c$, integration by parts yields
\begin{align*}
c^{2}\int_{\partial\Omega}
\bigl[V(u-\varphi)_{\nu}-(u-\varphi)V_{\nu}\bigr]d\sigma
&=
c^{2}\int_{\Omega}
\Bigl(V\Delta(u-\varphi)-(u-\varphi)\Delta V\Bigr)dv.
\end{align*}

Using $\Delta u=-1-nku$, $\Delta\varphi=-1-nk\varphi$, and $\Delta V=-nkV$, which follows from \eqref{Lapl_V}, we obtain
\begin{align*}
c^{2}\int_{\partial\Omega}
\bigl[V(u-\varphi)_{\nu}-(u-\varphi)V_{\nu}\bigr]d\sigma
&=
c^{2}\int_{\Omega}
\Bigl(V(-nk(u-\varphi))+nkV(u-\varphi)\Bigr)dv \\
&=0.
\end{align*}

Consequently,
\begin{align*}
\int_{\partial\Omega}u_{\nu}^{2}
\bigl[V(u-\varphi)_{\nu}-(u-\varphi)V_{\nu}\bigr]d\sigma
&=
\int_{\partial\Omega}(u_{\nu}^{2}-c^{2})
\bigl[V(u-\varphi)_{\nu}-(u-\varphi)V_{\nu}\bigr]d\sigma,
\end{align*}
and the desired identity follows.
\end{proof}

\begin{remark}\label{new:remSX}
When $M$ is the hyperbolic space or a hemisphere, with $V$ and $X$ as in Example \ref{space_forms}, and $\varphi\equiv-\frac{1}{nk}$, we have $\nabla V=-kX$, and identity \eqref{new:idB} becomes
\[
-2\int_{\Omega}Vu\,|\mathring{\nabla}^{2}u|^{2}dv=\int_{\partial\Omega}(u_{\nu}^{2}-c^{2})\Bigl[Vu_{\nu}+\frac{1}{n}\langle X,\nu\rangle\Bigr]d\sigma .
\]
This is identity (3.1) in \cite{scheuerxia} with $\phi\equiv1$ and $f=-u$ (there the ambient dimension is $n+1$), which was obtained by a different method and used to prove a quantitative stability result for Serrin's problem in space forms.
\end{remark}

\begin{remark}\label{existence}
In the setting of Theorem \ref{MainTh_2}, it follows from \eqref{Lapl_V} that
\[
\Delta V=-nkV.
\]
Therefore, if $V$ does not change sign, the problem \eqref{Dirichlet2} can be rewritten as
\begin{equation}\label{Dirichlet*}
\Delta u -\frac{\Delta V}{V} u = -1 \quad \text{in } \Omega,
\qquad
u = 0 \quad \text{on } \partial \Omega.
\end{equation}
The existence and uniqueness of solutions to \eqref{Dirichlet*} are guaranteed, for example, by \cite[Theorem 2.3]{foga}; see also \cite[Lemma 2.5]{li-xia}.
\end{remark}
In the next result, we show that, in conjunction with the overdetermined boundary condition \(u_{\nu}=-c\), the previous theorem immediately yields a rigidity statement. Although this result was proved via the \(P\)-function method in \cite[Corollary 1]{andrade2025} (see also \cite[Theorem 2]{freitas2024} for the closed conformal case), our approach follows directly from the integral identity established above (in space forms, compare with the stability result of \cite[Theorem 1.1]{scheuerxia}, which contains the corresponding rigidity statement), and Lemma \ref{new:obata} additionally identifies the metric of $\Omega$.

\begin{corollary}\label{cor:rigidity}
Let \( (M^n, g) \), $n\geq 3$, be an Einstein manifold with $\operatorname{Ric} = (n-1)k g,$ for some \( k \in \mathbb{R} \), and let \( \Omega\subset M \) be a domain with compact closure and \(C^{2}\)-boundary. Assume additionally that \( M \) is endowed with a conformal vector field \( X \) whose conformal factor satisfies $V>0$ in $\Omega$. If \( u \) is a solution of the overdetermined problem
\[
\begin{cases}
\Delta u + nku = -1 & \text{in } \Omega,\\
u>0 & \text{in } \Omega,\\
u=0 & \text{on } \partial\Omega,\\
u_{\nu}=-c & \text{on } \partial\Omega,
\end{cases}
\]
for some constant \(c>0\), then \( \Omega \) is isometric to a geodesic ball in the simply connected space form of constant curvature $k$, and \( u \) is the radial function \eqref{new:uexpl}.
\end{corollary}

\begin{proof}
Choosing in Theorem~\ref{MainTh_2} the constant $c$ given by the boundary condition, we have
\[
-2\int_{\Omega}Vu\,|\mathring{\nabla}^{2}u|^{2}dv=0.
\]
Since \(V>0\) and \(u>0\) in \(\Omega\), it follows that $\mathring{\nabla}^{2}u=0$, that is, $\nabla^{2}u=-\bigl(\frac1n+ku\bigr)g$ in $\Omega$. The conclusion then follows directly from Lemma~\ref{new:obata}.
\end{proof}


\section*{Data availability statement}
This manuscript has no associated data.

\section*{Conflict of interest statement}
On behalf of all authors, the corresponding author states that there is no conflict of interest.

\section*{Acknowledgments}
M. Andrade was partially supported by the Brazilian National Council for Scientific and Technological Development (CNPq, grants 408834/2023-4, 403869/2024-2, and 400078/2025-2) and FAPITEC/SE/Brazil (grant 019203.01303/2024-1). She also extends her gratitude to the Department of Mathematics at Princeton University, where this work was completed while she was a Visiting Fellow, for its hospitality. She is especially thankful to Ana Menezes for her warm hospitality and continuous encouragement.

A. Freitas was partially supported by the Brazilian National Council for Scientific and Technological Development (CNPq), grants 308141/2025-3 and 406078/2025-4 and Paraíba State Research Foundation (FAPESQ), Brazil, under grant 2453/2026. This study was financed, in part, by the São Paulo Research Foundation (FAPESP), Brasil, Process Number 2026/10743-0. He also wishes to thank IMECC/UNICAMP for the stimulating scientific environment and support provided during the development of part of this work. In particular, he is grateful to Lino Grama for his warm hospitality and constant encouragement.

The authors would like to thank Márcio Santos for his discussions on the subject of this paper and valuable suggestions.

\end{document}